\documentclass[12pt]{amsart}
\usepackage{amsmath,amssymb,amsbsy,amsfonts,amsthm,latexsym,mathabx,
            amsopn,amstext,amsxtra,euscript,amscd,stmaryrd,mathrsfs,
            cite,array,mathtools,enumerate}

\usepackage{url}
\usepackage[colorlinks,linkcolor=blue,anchorcolor=blue,citecolor=blue,backref=page]{hyperref}

\usepackage[norefs,nocites]{refcheck}
\usepackage{color}
\usepackage{float}

\hypersetup{breaklinks=true}

\usepackage[english]{babel}

\usepackage{mathtools}
\usepackage{todonotes}

\usepackage{enumitem}

\usepackage{mathtools}
\usepackage{todonotes}
\usepackage{url}
\usepackage[colorlinks,linkcolor=blue,anchorcolor=blue,citecolor=blue,backref=page]{hyperref}

\begin{document}

\newtheorem{theorem}{Theorem}
\newtheorem{lemma}[theorem]{Lemma}
\newtheorem{claim}[theorem]{Claim}
\newtheorem{cor}[theorem]{Corollary}
\newtheorem{prop}[theorem]{Proposition}
\newtheorem{definition}{Definition}
\newtheorem{question}[theorem]{Question}
\newtheorem{example}[theorem]{Example}
\newtheorem{remark}[theorem]{Remark}

\numberwithin{equation}{section}
\numberwithin{theorem}{section}

 \newcommand{\F}{\mathbb{F}}
\newcommand{\K}{\mathbb{K}}
\newcommand{\D}{\mathbb{D}}
\def\scr{\scriptstyle}
\def\\{\cr}
\def\({\left(}
\def\){\right)}
\def\[{\left[}
\def\]{\right]}
\def\<{\langle}
\def\>{\rangle}
\def\fl#1{\left\lfloor#1\right\rfloor}
\def\rf#1{\left\lceil#1\right\rceil}
\def\le{\leqslant}
\def\ge{\geqslant}
\def\eps{\varepsilon}
\def\mand{\qquad\mbox{and}\qquad}

\def\vec#1{\mathbf{#1}}

\def\vu{\vec{u}}

\def\bl#1{\begin{color}{blue}#1\end{color}} 

\def \F{{\mathbb F}}
\def \G{{\mathbb G}}
\def \K{{\mathbb K}}
\def \L{{\mathbb L}}
\def \N{{\mathbb N}}
\def \M{{\mathbb M}}
\def \Z{{\mathbb Z}}
\def \P{{\mathbb P}}
\def \Q{{\mathbb Q}}
\def \R{{\mathbb R}}
\def \C{{\mathbb C}}
\def \U{{\mathbb U}}
\def \Gal{{\mathrm{Gal}}}
 
\newcommand{\Disc}[1]{\mathrm{Disc}\(#1\)}
\newcommand{\Res}[1]{\mathrm{Res}\(#1\)}
\newcommand{\ord}{\mathrm{ord}}

\newcommand{\Norm}{\mathrm{Norm}}

\def\cA{{\mathcal A}}
\def\cB{{\mathcal B}}
\def\cC{{\mathcal C}}
\def\cD{{\mathcal D}}
\def\cE{{\mathcal E}}
\def\cF{{\mathcal F}}
\def\cG{{\mathcal G}}
\def\cH{{\mathcal H}}
\def\cI{{\mathcal I}}
\def\cJ{{\mathcal J}}
\def\cK{{\mathcal K}}
\def\cL{{\mathcal L}}
\def\cM{{\mathcal M}}
\def\cN{{\mathcal N}}
\def\cO{{\mathcal O}}
\def\cP{{\mathcal P}}
\def\cQ{{\mathcal Q}}
\def\cR{{\mathcal R}}
\def\cS{{\mathcal S}}
\def\cT{{\mathcal T}}
\def\cU{{\mathcal U}}
\def\cV{{\mathcal V}}
\def\cW{{\mathcal W}}
\def\cX{{\mathcal X}}
\def\cY{{\mathcal Y}}
\def\cZ{{\mathcal Z}}

\def\fra{{\mathfrak a}} 
\def\frb{{\mathfrak b}}
\def\frc{{\mathfrak c}}
\def\frd{{\mathfrak d}}
\def\fre{{\mathfrak e}}
\def\frf{{\mathfrak f}}
\def\frg{{\mathfrak g}}
\def\frh{{\mathfrak h}}
\def\fri{{\mathfrak i}}
\def\frj{{\mathfrak j}}
\def\frk{{\mathfrak k}}
\def\frl{{\mathfrak l}}
\def\frm{{\mathfrak m}}
\def\frn{{\mathfrak n}}
\def\fro{{\mathfrak o}}
\def\frp{{\mathfrak p}}
\def\frq{{\mathfrak q}}
\def\frr{{\mathfrak r}}
\def\frs{{\mathfrak s}}
\def\frt{{\mathfrak t}}
\def\fru{{\mathfrak u}}
\def\frv{{\mathfrak v}}
\def\frw{{\mathfrak w}}
\def\frx{{\mathfrak x}}
\def\fry{{\mathfrak y}}
\def\frz{{\mathfrak z}}
\def\Mult{\mathrm{Mult}}

\def\fE{{\mathfrak E}} 
\def\fA{{\mathfrak A}}

\def\ov#1{\overline{#1}}
\def \brho{\boldsymbol{\rho}}

\def \pf {\mathfrak p}

\def \Prob{{\mathrm {}}}
\def\e{\mathbf{e}}
\def\ep{{\mathbf{\,e}}_p}
\def\epp{{\mathbf{\,e}}_{p^2}}
\def\em{{\mathbf{\,e}}_m}
\def\n{\mathbf{n}}
\def\m{\mathbf{m}}

\newcommand{\sR}{\ensuremath{\mathscr{R}}}
\newcommand{\sDI}{\ensuremath{\mathscr{DI}}}
\newcommand{\DI}{\ensuremath{\mathrm{DI}}}

\newcommand{\h}{\mathrm{h}}
\newcommand{\di}{\mathrm{div}}
\newcommand\Gm{\G_{\mathrm{m}}}

\newcommand{\Orb}[1]{\mathrm{Orb}\(#1\)}
\newcommand{\aOrb}[1]{\overline{\mathrm{Orb}}\(#1\)}

\def \Nm{{\mathrm{Nm}}}

\newenvironment{notation}[0]{%
  \begin{list}%
    {}%
    {\setlength{\itemindent}{0pt}
     \setlength{\labelwidth}{1\parindent}
     \setlength{\labelsep}{\parindent}
     \setlength{\leftmargin}{2\parindent}
     \setlength{\itemsep}{0pt}
     }%
   }%
  {\end{list}}

\definecolor{dgreen}{rgb}{0.,0.6,0.}
\def\tgreen#1{\begin{color}{dgreen}{\it{#1}}\end{color}}
\def\tblue#1{\begin{color}{blue}{\it{#1}}\end{color}}
\def\tred#1{\begin{color}{red}#1\end{color}}
\def\tmagenta#1{\begin{color}{magenta}{\it{#1}}\end{color}}
\def\tNavyBlue#1{\begin{color}{NavyBlue}{\it{#1}}\end{color}}
\def\tMaroon#1{\begin{color}{Maroon}{\it{#1}}\end{color}}

\title[Skolem Problem for Parametric Recurrence Sequences]{On the Skolem problem and some related questions for parametric families of linear recurrence sequences}

\author[A. Ostafe] {Alina Ostafe}
\address{School of Mathematics and Statistics, University of New South Wales, Sydney NSW 2052, Australia}
\email{alina.ostafe@unsw.edu.au}

\author[I. E. Shparlinski] {Igor E. Shparlinski}
\address{School of Mathematics and Statistics, University of New South Wales, Sydney NSW 2052, Australia}
\email{igor.shparlinski@unsw.edu.au}



\begin{abstract} 
We show that in a parametric family of linear recurrence sequences $a_1(\alpha) f_1(\alpha)^n + \ldots  + a_k(\alpha) f_k(\alpha)^n$
with the coefficients $a_i$ and characteristic roots $f_i$, $i=1, \ldots,k$, given by rational functions over  some number field, for all 
but a set of elements 
$\alpha$  of bounded height  in the algebraic closure of $\Q$, the Skolem problem is solvable, and the existence of a zero in such a sequence
can be effectively decided.  
We also discuss several related questions.
\end{abstract} 

\maketitle


\section{Introduction}

\subsection{Motivation and background}
We recall that a linear recurrence sequence $\(u_n\)_{n=1}^\infty$ of order $k$ over a field $\K$ is a 
sequence which satisfies a relation of the form
\begin{equation}
\label{eq:LRS}
u_{n+k} = A_{k-1} u_{n+k-1} + \ldots + A_0 u_n,  \qquad n\ge 0, 
\end{equation}
with some constants $A_0, \ldots, A_{k-1}, u_0,\ldots,u_{k-1}\in \K$; we refer to~\cite{EvdPSW} for a background on recurrence sequences. 

It is well known by the celebrated Skolem-Mahler-Lech Theorem, see for example~\cite[Theorem~2.1]{EvdPSW}, that the set of zeros of a linear recurrence 
sequence $(u_n)_{n=0}^\infty$, that is, the set of $n\in \N$ such that $u_n=0$, is the union of a finite set with finitely many arithmetic progressions.  
 The famous {\it Skolem problem\/}  is a problem of decidability 
of the existence of such a zero. This problem remains widely open for most of the  
interesting fields $\K$ of characteristic zero, including $\K= \Q$ for $k\ge 5$, a brief outline 
of the current state of affairs is given by Sha~\cite{Sha}.  In particular, although there are 
very good, uniform in all parameters and depending only on $k$, bounds on the number of zeros, see~\cite{AmVia},
there are no effectively computable bounds on the index $n$ of a possible zero $u_n = 0$.

Here we address a parametric version of this problem for a family of linear recurrence sequences 
and  for all but finitely many values of the parameter for which the specialised sequence is not degenerate we give   a bound on the
largest possible zero. In particular, the Skolem problem is effectively decidable for all but a set of values 
of bounded height 
of the parameter $\alpha \in \ov \Q$.

More precisely, we now recall that if the characteristic polynomial 
$$
\Psi(Z) = Z^k - A_{k-1} Z^{k-1}  - \ldots -  A_0
$$ 
has $k$ distinct roots $\lambda_1, \ldots, \lambda_k$  then for any sequence  $(u_n)_{n=1}^\infty$ as in~\eqref{eq:LRS}
there are some $\mu_1, \ldots, \mu_k$ in an algebraic extension of $\K$ 
such that 
$$
u_n = \sum_{i=1}^k \mu_i \lambda_i^n, \qquad n \ge 0. 
$$

Linear recurrence sequences of this type are called {\it simple\/}. Some of the results on zeros of linear recurrence sequences 
in our arguments,
are also known  for arbitrary sequences,  for example, those of~\cite{Sha}, but some are known only under this condition, 
for example, those of~\cite{AMZ}.

In the case of $\K = \C$, the charactestic roots  $\lambda_i$, $i =1, \ldots, k$,  of the largest absolute value, that is, with 
$$
|\lambda_i| = \max\{|\lambda_1|, \ldots, |\lambda_k|\}
$$
are called {\it dominant\/} and play a major role in investigating various Diophantine properties of linear recurrence sequences.

We recall that a sequence  $(u_n)_{n=1}^\infty$ satisfying~\eqref{eq:LRS} is called {\it 
degenerate\/} if one of the ratios $\lambda_i/\lambda_j$, $1 \le i < j < k$, is a root of unity. We   call 
a sequence  {\it non-degenerate\/}  otherwise.

It is very well-known that the Skolem problem for a degenerate  sequence $u_n$ can be reduced to a series of 
the Skolem problems for finitely many non-degenerate sequences of the form $u_{nh+j}$, $ j=0, \ldots, h-1$,
where $h$ is bounded only in terms of the degree $[\Q(\lambda_1, \ldots, \lambda_k): \Q]$, see~\cite{BeMi}.
Hence  we are mostly interested in  non-degenerate sequences.

Here we study the case when both the coefficients $\mu_i$ and the roots $\lambda_i$
are rational functions  of a  parameter $\alpha \in \ov\Q$. This scenario is close to that of 
Amoroso,  Masser and  Zannier~\cite{AMZ}, see also~\cite[Proposition~2.2]{KMN}.  In particular, we also appeal to some of the  results 
from~\cite{AMZ}, however our method is based on a different argument.

More precisely, given  two vectors of rational functions 
$$
\vec{a}= \(a_1,\ldots,a_k\), \, \vec{f} = \(f_1,\ldots,f_k\)\in\ov\Q(X)^k,
$$
over the algebraic closure $\ov\Q$ of $\Q$ 
 we consider linear recurrence sequences of the form
\begin{equation}
\label{eq:Fn}
F_n(X) = \sum_{i=1}^k a_i f_i^n, \qquad n \ge 0. 
\end{equation}
We give a bound for the largest zero in (all but finitely many) specialisations of this sequence such that the 
new sequence is not degenerate, see  Theorem~\ref{thm:LRS-Zeros Qbar}.  This is based on recent results of Pakovich and the second author~\cite{PaSh} 
about finiteness of points on curves that lie on the unit circle, a bound for the largest zero in a linear recurrence 
with at most two dominant roots due to Sha~\cite{Sha} and a  result of Amoroso,  Masser and  Zannier~\cite[Theorem~1.5]{AMZ} giving an upper bound for the height of zeros of polynomials of the form~\eqref{eq:Fn}. We then couple our result with a general  ABC theorem for polynomials~\cite[Theorem~12.4.4]{BoGu}  to give a lower bound on the degree of the 
splitting field of $F_n$ over $\Q$     
when $a_i,f_i$ are all polynomials, see  Corollary~\ref{cor:split field}, which gives an explicit form of~\cite[Example~2]{AMZ}.

We also look at complex numbers which are zeros of two functions of the form~\eqref{eq:Fn}, 
see Theorem~\ref{thm:gcd prod pow}.   
This is closely related
 to studying the greatest common divisors of elements of two sequences defined by such functions.  
  This point of view has  been introduced by Ailon and Rudnick~\cite{AR} for polynomials of the form $f^n-1$ and $g^m-1$, $n,m \ge 1$, and further developed in recent works~\cite{CZ,Lev,LevWan,Ost,PaSh} for different other sequences.
In turn, this builds on number field analogues of this problem, initiated by Bugeaud, Corvaja and Zannier~\cite{BCZ} for sequences of the form $a^n-1$ and $b^m-1$, $n,m\ge 1$, where $a,b\in\Z$ are multiplicatively independent, and further extended in many works, including generalisations to $S$-units~\cite{CZ05}.

When we restrict to the set $\U$  of all roots of unity only, we can say more about the arithmetic structure of $F_n(\alpha)$, $n\ge 1$, as defined in~\eqref{eq:Fn}, that is, for all but finitely many $\alpha\in \U$ we prove that the set of $n\in \N$ such that $F_n(\alpha)$ is a perfect $m$th power either contains an arithmetic progression, or it is finite, see Theorem~\ref{thm:LRS-PowU}. We conclude the paper with an independent result, Theorem~\ref{thm:LRS spec inf}, which says that a sequence of rational functions is a linear recurrence sequence if and only if  for uncountably many $\alpha\in\C$  the specialisations of the sequence have the same property.

\section{Main Results}

\subsection{Skolem problem for parametric families} 

We use $M_\K$ to denote a complete set of inequivalent absolute values
on a number field~$\K$, normalized so that the absolute Weil height
$h:\ov{\Q}\to[0,\infty)$ is defined via $v$-adic valuations   $\|\alpha\|_v$
as follows 
$$
h(\alpha) = \sum_{v\in M_\K} \dfrac{[\K_v: \mathbb{Q}_v]}{[\K:\mathbb{Q}]}\log \max\left \{ \|\alpha\|_v, 1 \right\} ,
$$
where $\K$ is any number field containing $\alpha$ and $\K_v$ is the completion of $\K$ with respect to the absolute value $v$. See~\cite{BoGu,HinSil,Silv,Zan} for further details on absolute
values and height functions.

To formulate our results we need to recall that a {\it finite Blaschke product} is a rational function $B(z) \in \C(z)$ of the form 
\begin{equation}
\label{eq:B}
B(z)=\zeta\prod_{i=1}^n\left(\frac{z-a_i}{1-\bar{a_i}z}\right)^{m_i},
\end{equation}  
where $a_i$ are complex numbers in the open unit disc $\{z \in \C:~|z| <1\}$,
the exponents $m_i$, $i=1, \ldots, n$, are positive integers, and $\left| \zeta\right |=1$. 

A rational function 
$Q(z)$ of the form $Q(z) = B_1(z)/B_2(z)$, where 
$B_1$ and $B_2$ are finite Blaschke products,  is called a {\it quotient of finite Blaschke products}. 

We say that a pair of rational functions $\(g_1(X),  g_2(X)\) \in \C(X)^2$ is {\it exceptional\/} if 
they can be decomposed as 
\begin{equation}
 \label{eq:gBg} 
g_1=Q_1\circ g \mand g_2=Q_2\circ g
\end{equation}
 for some  quotients of finite Blaschke products $Q_1$ and $Q_2$ and 
rational function $g$. Otherwise we say that  $\(g_1(X),  g_2(X)\)$ is
a {\it non-exceptional\/}  pair.

We recall that Blaschke products preserve the unit circle and thus come naturally in 
our investigation of  $\alpha\in \ov \Q$ for which the corresponding specialisation of $F_n(X)$
given by~\eqref{eq:Fn} has at least three characteristic  roots with the same modulus, that is, with
$$
|f_r(\alpha)|=|f_s(\alpha)|=|f_t(\alpha)|,
$$
for some $1 \le r< s < t \le k$, which implies 
$$
\frac{|f_s(\alpha)|}{|f_r(\alpha)|}=\frac{|f_t(\alpha)|}{|f_r(\alpha)|}=1.
$$

We see that if $g_1$ and $g_2$ are polynomials and thus have no finite poles,  relations~\eqref{eq:gBg}  
are possible only if $n=1$ and $a_1 = 0$ in~\eqref{eq:B}, in which case we obtain 
$$g_1= g^{m_1} \mand   g_2=g^{m_2}$$
for some $g\in \C[X]$ and integers $m_1, m_2 \ge 0$.   

Since as we have mentioned, the Skolem problem for  degenerate sequences
is easily reducible to the case of non-degenerate sequences, it motivates us to define
 for the parametric family~\eqref{eq:Fn} 
the exceptional set of $\alpha \in  \ov\Q$ such that $f_i(\alpha)/f_j(\alpha)$ is not a root of unity for some $1 \le i <j\le k$.
It is also convenient to eliminate the roots of  $a_i(X)$ and $f_i(X)$, $1 \le i \le k$.
That is, we define the set 
\begin{align*} 
\fE_{\vec{a},\vec{f}}  = \{\alpha\in \ov \Q:~ f_i(\alpha)/f_j(\alpha)~\text{is a root of unity for some $1 \le i <j\le k$}& \\
 \text{or 
$a_i(\alpha)=0$ or $f_i(\alpha) = 0$ for some  $1 \le i \le  k$}\}&.
\end{align*} 

Using~\cite[Theorem~3.11]{Silv} we immediately derive that elements of $\fE_{\vec{a},\vec{f}}$ are of bounded 
height. 

Our first result is the following:

\begin{theorem}
\label{thm:LRS-Zeros Qbar} Let $a_i, f_i \in \ov\Q(X)$, $i =1, \ldots, k$, be nonzero rational functions of degree at most $d$ such that   $f_i/f_j$ is non-constant for any $1 \le i <j \le k$. 
We also assume that 
 for any $1 \le r<s <t  \le k$, the pairs of rational functions $(f_s/f_r, f_t/f_r)$ are non-exceptional. 
 Let the sequence $(F_n)_{n=0}^\infty$ be defined by~\eqref{eq:Fn}. 
Then for all but at most $2d^2 k(k-1)(k-2)/3$
 elements $\alpha \in  \ov\Q \setminus \fE_{\vec{a},\vec{f}}$ 
 any zero $n\in \N$ of the
equation 
\begin{equation}
\label{eq:zero LRS}
F_n(\alpha)  = 0
\end{equation}
satisfies 
$$
n \le \exp\(C D^4_\alpha\),
$$
where $D_\alpha$ is the degree of the Galois closure of $\Q(\alpha)$ over $\Q$
and    $C$ is an effective constant which depends only on  $a_1, f_1, \ldots, a_k, f_k$.    
\end{theorem} 

\begin{remark} 
We note that the condition that $f_i(\alpha)/f_j(\alpha)$ are  not roots of unity for any $i\ne j$ is necessary for any  bound 
 on $n$ 
as otherwise  the sequence can have infinitely many zeros. However, if this is the case, then by the celebrated Skolem-Mahler-Lech Theorem, see for example~\cite[Theorem~2.1]{EvdPSW}, the set of zeros is the union of a finite set with finitely many arithmetic progressions, and by~\cite{BeMi} these arithmetic progressions can be effectively determined, while our method allows to estimate the elements in the remaining finite set. 
\end{remark}

It is natural to ask whether  a generalisation of Theorem~\ref{thm:LRS-Zeros Qbar} to parametric $S$-unit equations  is possible, see Section~\ref{sec:comm} for more details and an exact question. 

Theorem~\ref{thm:LRS-Zeros Qbar},  immediately implies the following: 

\begin{cor}
\label{cor:LRS-Zeros AD} 
Let $a_i, f_i \in \ov\Q[X]$, $i =1, \ldots, k$, be  as in Theorem~\ref{thm:LRS-Zeros Qbar}.
 Let
 $\cA_D$ be the union of all   Galois fields of degree at most $D$ over $\Q$.
 Then there is an  effectively computable constant $C_0$ which depends only on $a_1, f_1, \ldots, a_k, f_k$
 such that for all but  at most $12 dD^2k^2 + 2d^2k^3/3$ 
 elements $\alpha \in  \cA_D$, any zero $n\in \N$ of the
equation~\eqref{eq:zero LRS} 
satisfies 
$$
n \le \exp(C_0 D^4).
$$ 
\end{cor}

Combining Theorem~\ref{thm:LRS-Zeros Qbar} with a  result of Amoroso,  Masser and  Zannier~\cite[Theorem~1.5]{AMZ} (see Lemma~\ref{lem:AMZ}), we also derive the following. 

\begin{cor}
\label{cor:IrredFact}  
Let $a_i, f_i \in \ov\Q[X]$, $i =1, \ldots, k$,  be  as in Theorem~\ref{thm:LRS-Zeros Qbar}.
  Then there exists a  finite  set $\cE\subseteq \ov\Q$  such that for all  roots $\alpha$ of $F_n$, $n\ge 1$, with $\alpha \not \in \cE \cup   \fE_{\vec{a},\vec{f}}$
  the degree $D_\alpha$ of the smallest Galois extension over $\Q$ containing $\alpha$ satisfies
  $$
D_\alpha \ge c (\log n)^{1/4},
$$
  where $c$ is an effective constant depending  only on  $a_1, f_1, \ldots, a_k, f_k$. 
  \end{cor}  

Finally  Theorem~\ref{thm:LRS-Zeros Qbar} coupled with a  result of  Fuchs  and  Peth\"o~\cite[Corollary~3.1]{FuPe} (see Lemma~\ref{lem:fuchs})  yields the following result. 

\begin{cor} \label{cor:split field}
Let $a_i, f_i \in \ov\Q[X]$, $i =1, \ldots, k$,  be   as in Theorem~\ref{thm:LRS-Zeros Qbar}
 and such that  $\gcd(a_1f_1, \ldots, a_kf_k)=1$. 
 Then the splitting field $\L_n$ of the polynomial $F_n$ defined by~\eqref{eq:Fn}  is
of degree 
$$
[\L_n:\Q] \ge c_0 (\log n)^{1/4},
$$
where $c_0$ is an effective constant depending  only on  $a_1, f_1, \ldots, a_k, f_k$. 
\end{cor}  

Note that Corollary~\ref{cor:split field} provides  an explicit version of the claim  that 
$$
[\L_n:\Q] \to \infty
$$
as $n \to \infty$ given in~\cite[Example~2]{AMZ}.

The bound for $n$ in Theorem~\ref{thm:LRS-Zeros Qbar} depends on the degree of the specialisation $\alpha$. We would like to obtain a bound which is independent of $\alpha$, when we restrict $\alpha$ to special subsets of $\ov\Q$, such as the set of all roots of unity. This in particular would imply that the set 
$$
\{\alpha \in \ov\Q :~\alpha^n=1,\ F_m(\alpha)=0 \textrm{ for some $n,m\ge 1$}\},
$$
is finite, where $F_m$ is defined by~\eqref{eq:Fn}.  

More generally, since $G_n(X)=X^n-1$, $n\ge 1$,  also defines a linear recurrence sequence, we would like to generalise the above situation and formulate the following problem.

\begin{question}
\label{quest:CommZero}
Given two simple  linear recurrence sequences $(F_n)_{n=1}^\infty$, $(G_m)_{n=1}^\infty$ over $\C(X)$, prove that, under certain conditions,  the set
\begin{equation}
\label{eq:L}
\cL(F_n,G_m)=\{\alpha\in\C :~F_n(\alpha)=G_m(\alpha)=0 \textrm{ for some $n,m\ge 1$}\}
\end{equation}
is finite. 
\end{question}

This would extend the result of~\cite{AMZ} which already gives 
the boundedness of  heights of  $\alpha\in\ov\Q$ that satisfy only one such relation over $\ov\Q$

We answer Question~\ref{quest:CommZero} for binary sequences 
in Section~\ref{sec:gcd} below, which is a direct consequence of~\cite[Theorem 2]{BMZ}.
We also prove a more general result which holds only over $\ov\Q$, see Theorem~\ref{thm:gcd prod pow}.

\subsection{Perfect powers in specialisations at roots of unity}

We now restrict $\alpha$ to the set $\U$ of all roots of unity. In this case, for all but finitely many $\alpha \in \U$ we can 
effectively answer some other questions about $F_n(\alpha)$ besides vanishing. We illustrate this approach on the case 
of perfect powers. 

We use  $\Z_{\ov\Q}$ to denote the set of all algebraic integers in $\ov\Q$.

We say that   $\vartheta \in \Z_{\ov\Q}$ is a perfect $m$th power if for some   $\rho \in \Q(\vartheta)$ we have $\vartheta = \rho^m$. 

For a polynomial $f\in \C[X]$, we use $\ov f$ to denote the complex conjugate polynomial, 
that is, the polynomial obtained from $f$ by conjugating all its coefficients.

\begin{theorem}
\label{thm:LRS-PowU} Let $a_i, f_i \in \Z_{\ov\Q} [X]$, $i =1, \ldots, k$, be nonzero  polynomials of degree at most $d$ such that    for any $1 \le r<s <t  \le k$, the pairs of rational functions $(f_s/f_r, f_t/f_r)$ are non-exceptional. We also assume that for $1\le i\ne j\le k$ the polynomials $f_i(X)\ov{f_i}(Y) - f_j(X)\ov{f_j}(Y)$  do not have any factor of the form 
$X^rY^s-u$ or $X^r-uY^s$ where  $u$ is a root of unity. 
 Let the sequence of polynomials $\(F_n\)_{n=0}^\infty$ be defined by~\eqref{eq:Fn}.  
Then  for all but at most $d^2(2k^3/3 + 22k^2)$  
 elements $\alpha \in   \U \setminus \fE_{\vec{a},\vec{f}}$ for any $m \ge 1$
 the set of   $n\in \N$ such that $F_n(\alpha) $ is a perfect $m$th  power either contains an infinite arithmetic progression $m\ell +r$, $\ell  =1,2, \ldots$,
 with some $r \in \{0, \ldots, m-1\}$ or is finite and its size can be effectively bounded.  
 \end{theorem}  
 
 \begin{remark} 
 We note that by a result of Lang~\cite{Lang}   the set $\U \cap  \fE_{\vec{a},\vec{f}}$ is in fact finite, thus the number of exceptions in Theorem~\ref{thm:LRS-PowU} is finite.
 \end{remark}  
 
\subsection{Greatest common divisors of binary linear recurrence sequences}
\label{sec:gcd}

We note that the finiteness conclusion in Question~\ref{quest:CommZero} would imply that there are finitely many $\alpha\in\ov\Q$ such that 
$$
(X-\alpha) \mid \gcd(F_n(X),G_m(X))
$$
for some $n,m\ge 1$, 
where by the {\it greatest common divisor of two rational functions\/}, we mean the 
greatest common divisor  of their numerators.

Studying the greatest common divisor of sequences of polynomials has been initiated in~\cite{AR} for sequences $F_n=f^n-1$ and $G_m=g^m-1$, $n,m\ge 1$, where $f,g\in\C[X]$, and recently extended in several directions in~\cite{CZ,Lev,LevWan,Ost,PaSh}.

We note that in the case of  simple binary linear recurrence sequences defined over $\C(X)$, under a multiplicative independence condition, a finiteness result follows directly from~\cite[Theorem~1.3]{Ost}. In fact here we consider a more general scenario, which however holds only when the  rational functions involved are defined over $\ov\Q$.

More precisely, for a finitely generated group $\Gamma\subseteq \ov\Q^*$, we define the {\it division group of $\Gamma$} to be the set
$$
\Gamma^\di=\{\alpha\in\ov\Q :~\alpha^n \in \Gamma \textrm{ for some $n\ge 1$}\}.
$$
Then we have  the following result.

We say that $\psi_1,\ldots,\psi_t\in \C(X)$ are
{\it  multiplicatively independent modulo a set  $\cS \subseteq\C$\/} if for
all non-zero vectors $\(h_1, \ldots, h_t\) \in \Z^t$  one has
$$
\psi_1(X)^{h_1} \ldots \psi_t(X)^{h_t} \not \in \cS.
$$ 

\begin{theorem}
\label{thm:gcd prod pow} 
Let  $\Gamma\subseteq \ov\Q^*$ be a finitely generated subgroup. Let $f_1,\ldots,f_r,g_1,\ldots,g_s\in\ov\Q(X)$ be multiplicatively independent modulo  $\Gamma$. 
There exists a polynomial  $h\in\ov\Q[X]$ depending only on $f_i,g_j$ and the generators of 
$\Gamma$  such that 
 that for any nonzero vectors $(m_1,\ldots,m_r)\in\N^r$,  $(n_1,\ldots,n_s)\in\N^s$  and any $u,v \in \Gamma^\di$ one has
$$
\gcd\(\prod_{i=1}^rf_i(X)^{m_i}-u,\prod_{j=1}^sg_j(X)^{n_j}-v \)\mid h.
$$ 
\end{theorem}   

\begin{remark}
Taking $r=s=2$, $\Gamma=\{1\}$, $m_1=n_1=1$ and $u=v=-1$ in Theorem~\ref{thm:gcd prod pow}, one has the sequences $(F_n)_{n=1}^\infty$, $(G_m)_{n=1}^\infty$  defined by
\begin{align*} 
&F_n(X)=f_1(X)f_2(X)^n+1, \quad n\ge 1,\\
 &G_m(X)=g_1(X)g(_2X)^m+1, \quad m\ge 1.
\end{align*}
In this case the set $\cL(F_n,G_m)$, defined by~\eqref{eq:L}, is finite and its cardinality is uniformly bounded only in terms of the functions $f_1,f_2,g_1,g_2$. In fact, this conclusion holds also when $f_i,g_i\in\C(X)$, $i=1,2$. Indeed, the proof follows  directly from~\cite[Theorem~1.3]{Ost}, which is the same as the proof of Theorem~\ref{thm:gcd prod pow} but instead of using Lemma~\ref{lem:multdep-rat} below, one uses~\cite[Theorem]{BMZ08} (which in turn generalises~\cite[Theorem 2]{BMZ}).
See also~\cite{BiLu}
for some effective results for  rational functions over $\Q$ and $G_m = X^m -1$. 
\end{remark}

As a direct consequence of Theorem~\ref{thm:gcd prod pow} we have the following  result, which we hope to be  of independent interest.

\begin{cor}
\label{cor:gcd poly} 
Let  $\Gamma\subseteq \ov\Q^*$ be a finitely generated subgroup. Let $f_1,\ldots,f_r,g_1,\ldots,g_s\in\ov\Q(X)$ be multiplicatively independent modulo  $\Gamma$. 
Then there exists a polynomial $H\in\ov\Q[X]$ such that for any nonzero vectors $(m_1,\ldots,m_r)\in\N^r$,  $(n_1,\ldots,n_s)\in\N^s$ and  any polynomials $F,G \in\ov\Q[X]$ of degree at most $d\ge 1$ with roots from $\Gamma^\di$, one has
$$
\gcd\(F\(\prod_{i=1}^rf_i(X)^{m_i}\),G\(\prod_{j=1}^sg_j(X)^{n_j}\)\) \mid H.
$$
\end{cor}

\subsection{Specialisations of sequences of rational functions}

Finally, we show that a sequence of rational functions $F_n(X) \in\C(X)$, $n\ge 0$,
satisfies a linear recurrence relation if and only if this is also true 
for specialisations $F_n(\alpha)$, $n\ge 0$, for a sufficiently ``massive'' set of $\alpha \in \C$.

\begin{theorem}
\label{thm:LRS spec inf} Let $\(F_n\)_{n=0}^{\infty}$ be an infinite sequence of rational functions $F_n\in\C(X)$, $n\ge 0$, and let $K\ge 1$. Assume there exist infinitely many $\alpha\in\C$ such that each $F_n(\alpha)$ is well defined and $\(F_n(\alpha)\)_{n=0}^{\infty}$ is a  linear recurrence of order at most $K$. Then $\(F_n\)_{n=0}^{\infty}$ is a  linear recurrence  sequence over $\C(X)$ of order at most $K$.
\end{theorem}

Theorem~\ref{thm:LRS spec inf} implies  that  if $\(F_n(\alpha)\)_{n=0}^{\infty}$ is a   linear recurrence  sequence for
 uncountably  many   $\alpha\in\C$ (without any restriction on the order), then $\(F_n\)_{n=0}^{\infty}$ is a  
 linear recurrence  sequence in $\C[X]$. Indeed, since the set of possible orders is countable, 
 at least one order must  occur for infinitely  many $\alpha \in \C$ (in fact uncountably many $\alpha$).
 On the other hand, just a condition of  the infinitude of such $\alpha\in\C$ is not sufficient. 
For example, the sequence 
$$
F_n(X) = X^{2^{n^2}+n} + X^{6^n} + X^{2^{\fl{\sqrt{n}}}} + 3^n, \qquad n = 0,1, \ldots,
$$
does not satisfy any linear recurrent relation (as the degree grows too fast). However it is a   linear recurrence  sequence 
for  specialisations at any  root of unity of order $2^m$, $m =0, 1, \ldots$. 

\section{Preliminaries} 

\subsection{Intersections of level curves of rational functions}

Our approach depends on a generalisation  of a result   of 
Ailon and Rudnick~\cite[Theorem~1]{AR} which is given in~\cite[Theorem~2.2]{PaSh}. 

Recall the definition of  non-exceptional pairs of rational functions, that is, pairs which do not satisfy  relations of the form~\eqref{eq:gBg}.

\begin{lemma}
 \label{lem:level curves}
Let $g_1(X), g_2(X) \in \C(X)$ be complex rational functions of degrees $n_1$ and $n_2$, respectively. Then 
$$
\# \{z \in \C:~\left|g_1(z)\right | =\left| g_2(z)\right | =1\} \le (n_1+n_2)^2,
$$
unless $\(g_1(X), g_2(X)\)$ is exceptional.
\end{lemma}

As we have mentioned if $g_1$ and $g_2$ are polynomials then the
conclusion of Lemma~\ref{lem:level curves} holds except when 
$$g_1= g^{m_1} \mand   g_2=g^{m_2}$$
for some $g\in \C[X]$ and integers $m_1, m_2 \ge 0$.

\subsection{Intersection of parametric curves with division groups} 

The proof of Theorem~\ref{thm:gcd prod pow} is based on the following result which  is a direct consequence of Maurin's result~\cite[Th{\'e}or{\`e}me~1.2]{Mau};
see also~\cite{BHMZ} for an effective proof of~\cite[Th{\'e}or{\`e}me~1.2]{Mau}.

\begin{lemma}
\label{lem:multdep-rat}
Let  $\Gamma\subseteq \ov\Q^*$ be a finitely generated subgroup. Let $f_1,  \ldots,$ $ f_n \in \ov\Q(X)$ be multiplicatively independent  modulo $\Gamma$.
Then, there are only finitely many $\alpha\in \ov\Q$ such that
there exist two linearly independent vectors $(k_1,\ldots,k_n)$, $(\ell_1,\ldots,\ell_n)\in\Z^n$ such that 
$$
f_1(\alpha)^{k_1}\cdots f_n(\alpha)^{k_n},\ f_1(\alpha)^{\ell_1}\cdots f_n(\alpha)^{\ell_n}\in \Gamma^\di.
$$
\end{lemma}

\begin{proof} 
The proof follows directly from a result of Maurin~\cite[Th{\'e}\-or{\`e}me~1.2]{Mau} 
(see also~\cite[Theorem (Maurin)]{BHMZ}).
Indeed, let $\gamma_1,\ldots,\gamma_r \in \Gamma$ be the generators of the group $\Gamma$,
and thus, they are multiplicatively independent elements. We define the parametric curve
$$
\cC=\{(f_1(\alpha),\ldots,f_n(\alpha),\gamma_1,\ldots,\gamma_r) :\,  \alpha \in \ov\Q\}\subset \Gm^{n+r}.
$$

For an element $\alpha \in \overline{\Q}$,  assume that there exist two linearly independent vectors $(k_1,\ldots,k_n)$, $(\ell_1,\ldots,\ell_n)\in\Z^n$ such that both
$$
f_1(\alpha)^{k_1}\cdots f_n(\alpha)^{k_n},\ f_1(\alpha)^{\ell_1}\cdots f_n(\alpha)^{\ell_n}\in \Gamma^\di.
$$ 
Then, this leads to two multiplicative relations:
\begin{align*}
& f_1(\alpha)^{sk_1}\cdots f_n(\alpha)^{sk_n}\gamma_1^{k_{n+1}}\cdots \gamma_r^{k_{n+r}} = 1,  \\
& f_1(\alpha)^{t\ell_1}\cdots f_n(\alpha)^{t\ell_n}\gamma_1^{\ell_{n+1}}\cdots \gamma_r^{\ell_{n+r}} = 1.
\end{align*}
for some integer vectors $(k_{n+1},\ldots,k_{n+r})$, $(\ell_{n+1},\ldots,\ell_{n+r})\in\Z^{n+r}$ and some positive integers $s,t$.

Since $f_1,\ldots,f_n$ are multiplicatively independent modulo $\Gamma$, we know that the functions $f_1,\ldots,f_n,\gamma_1,\ldots,\gamma_r$ are multiplicatively independent.
We also note that the two vectors $(sk_1,\ldots, sk_n, k_{n+1}, \ldots, k_{n+r})$, 
$(t\ell_1,\ldots,t\ell_n, \ell_{n+1}, \ldots \ell_{n+r})$ remain linearly independent.
Then, the conclusion follows directly from~\cite[Th{\'e}or{\`e}me~1.2]{Mau}.
\end{proof}

\subsection{Zeros of linear recurrence sequences with at most two dominant roots}

We say that a root $\lambda$ of a polynomial $\psi(Z) \in \C[Z]$ is dominant 
if $|\lambda| \ge |\rho|$ for any other root $\rho$ of $\psi$. Clearly $\psi$ may have 
up to $\deg \psi$ dominant roots. We are mostly interested in the case of one or two dominant
roots.  

We note the following  result which represents a simplified combination 
 of two results  (for $\nu=1$ and $\nu=2$) of Sha~\cite[Theorems~1.1 and~1.2]{Sha}. 

\begin{lemma}
 \label{lem: 1-2 dom. root} 
Let $\psi(X) \in \ov \Q[X]$ be a monic square-free polynomial of degree $k$ and with $\nu \le 2$ dominant roots. If $\nu =2$, we also impose that the ratio of the two dominant roots is not a root of unity. 
Then for any linear recurrence sequence $(u_n)_{n=0}^\infty$  given by 
$$
u_n = \sum_{i=1}^k \alpha_i \lambda_i^n , \qquad n \ge 0,
$$
with characteristic polynomial $\psi$ having roots $\lambda_i$, $i=1,\ldots,k$, if   
$$
u_n  = 0
$$
then
$$
n \le  \exp\(C_1 D^4 \(\max_{i=1, \ldots, k} h\(\alpha_i\)+1\)\), 
$$
where $C_1$ is an effective constant which depends only on $k$, $D=[\K:\Q]$ is the degree  of the smallest 
Galois field $\K$ containing the coefficients of $\psi$ and the initial values $u_0, \ldots, u_{k-1}$. 
\end{lemma}

\subsection{Zeros of linear recurrence sequences in families}

We now recall a special case of a  result of Amoroso,  Masser and  Zannier~\cite[Theorem~1.5]{AMZ}. 

\begin{lemma}
 \label{lem:AMZ}
 Let $f_i \in \ov\Q(X)$, $i =1, \ldots, k$, be nonzero rational functions such that $f_s/f_r$ is non-constant 
for any  $1 \le r<s  \le  k$. There exists an  effectively computable constant $C_2$, which depends on $f_1, \ldots, f_k$ such that if
for any $b_1, \ldots, b_k \in \ov\Q$ not all zero, any $n \ge C_2$  and any $\alpha \in \ov \Q\setminus \fE_{\vec{a},\vec{f}}$, one has 
$$
\sum_{i=1}^k b_i  f_i(\alpha)^n = 0,
$$
then 
$$
h(\alpha) \le \frac{k\max\{h(b_1), \ldots,h(b_k) \}}{n}+ C_2.
$$ 
\end{lemma}

We remark that in~\cite[Theorem~1.5]{AMZ} the bound is given in terms of the height 
of the projective vector $(b_1: \ldots: b_k) \in \P^{k-1}\(\ov\Q\)$ which is bounded by the  maximum
used in  Lemma~\ref{lem:AMZ}. 
Furthermore, in~\cite[Theorem~1.5]{AMZ} only one ratio is 
assumed to be non-constant, however there is an additional request of non-vanishing of subsums 
in the sum of Lemma~\ref{lem:AMZ} (under our condition on  $f_s/f_r$ one can simply consider 
the shortest vanishing subsum). 

We also note that the effectiveness is not explicitly 
stated in~\cite[Theorem~1.5]{AMZ} however it is discussed after the formulation of~\cite[Theorem~1.5]{AMZ}.

\begin{cor}
\label{cor:AMZ}
Let $a_i,f_i \in \ov\Q(X)$, $i =1, \ldots, k$, be nonzero rational functions such that $f_s/f_r$ is non-constant 
for any  $1 \le r<s  \le  k$. There exists an  effectively computable constant $C_3$, which depends on $a_1,\ldots,a_k,f_1, \ldots, f_k$ such that if
for any $n\ge C_3$ and any $\alpha \in \ov \Q$ one has 
$$
\sum_{i=1}^k a_i(\alpha)  f_i(\alpha)^n = 0,
$$
then 
$$
h(\alpha) \le C_3.
$$
\end{cor}

\begin{proof} If  $\alpha \in  \fE_{\vec{a},\vec{f}}$, as we have mentioned, this follows from~\cite[Theorem~3.11]{Silv}.

We can now assume that $\alpha \in \ov \Q\setminus \fE_{\vec{a},\vec{f}}$.
We can also assume that $n > C_2$, where $C_2$ is as in Lemma~\ref{lem:AMZ}, 
since otherwise $\alpha$ is a root of an equation of bounded degree and height.

We now recall that for any $\alpha \in \ov \Q$ and $a\in \ov\Q(X)$ 
we have 
$$
h\(a(\alpha)\) \le h\(\alpha\)  \deg a + C_4,
$$
where $C_4$ is an effectively computable constant that depends only on $a$, see~\cite[Theorem~3.11]{Silv}. 
Hence for 
$$
n \ge 2 k  \max\{\deg a_1, \ldots, \deg a_k\} 
$$
the bound of Lemma~\ref{lem:AMZ} becomes
\begin{align*}
h(\alpha) & \le \frac{k \max\{h\((a_1(\alpha)\), \ldots, h\(a_k(\alpha)\)\} }{n}+ C_2\\
 & \le \frac{k  \max\{\deg a_1, \ldots, \deg a_k\}  h\(\alpha\) + k C_4}{n}+ C_2\\
  & \le \frac{1}{2} h\(\alpha\) +  C_4/2+ C_2. 
 \end{align*}
 Hence $h(\alpha)$ is bounded only in terms of  $a_1, f_1, \ldots, a_k, f_k$. 
Choosing 
$$
C_3= 2k\max\{\deg a_1, \ldots, \deg a_k, C_2\},
$$
the result follows. 
\end{proof}

\subsection{Polynomial ABC theorem}  
We need the following generalisation~\cite[Theorem~12.4.4]{BoGu} of the ABC Theorem for polynomials proved first by Stothers~\cite{St}, and then independently by Mason~\cite{Mas0} and Silverman~\cite{Si84}. This is a special case of a more general result for $S$-units in function fields due to Voloch~\cite{Vo} and  then also  Brownawell and Masser~\cite{BrMa}.

\begin{lemma}
\label{lem:abc}
Let $g_1,\ldots,g_m\in\ov\Q[X]$, $m\ge 3$, be such that $g_1,\ldots,g_m$ have no common zero in $\ov\Q$. Assume that
$$
g_1(X)+\cdots+g_m(X)=0
$$
and also that no proper subsum 
$$
\sum_{i \in \cI} g_i(X), \qquad \text{with} \quad \cI \subseteq \{1, \ldots, m\}, \ 1 \le \#\cI< m, 
$$
 vanishes identically. Then
$$
\max_{i=1,\ldots,m} \deg g_i\le \frac{(m-1)(m-2)}{2} \max\left\{\deg\mathrm{rad}\(\prod_{i=1}^m g_i\)-1,0\right\}.
$$
\end{lemma}

We also need the following result on the multiplicity of roots of certain rational functions, see~\cite[Corollary 2.10]{Ost}. For a rational function  $h\in\C(T)$, we denote by $\Mult(h)$ the largest multiplicity of the zeros of $h$ and by $Z(h)$ the set of zeros of $h$ in $\C$, 
respectively. 

\begin{lemma}
\label{lem:ABCrat}
Let  $h_i=f_i/g_i$, $f_i,g_i\in\C[T]$, $i=1,\ldots,n$, with 
$$
Z(f_1\cdots f_m)\cap Z(g_1\cdots g_m)=\emptyset.
$$ 
Then, for all integers $n_1,\ldots,n_{\ell}\ge 0$ and $a\in\C^*$, we have 
$$
\Mult\(h_1^{n_1}\cdots h_{\ell}^{n_{m}}-a\)\le \sum_{i=1}^{m} \(\deg f_i+\deg g_i\).
$$
\end{lemma}

\subsection{Zeros of linear recurrence sequences in function fields} 
It is well known that the Skolem Problem is settled in the case of linear recurrence sequences over function fields. An effective (but not explicit) bound on the largest zero in such a sequence is given in~\cite[Corollary~3.1]{FuPe}. We present it in a simplified form as needed for our setting.

\begin{lemma}
\label{lem:fuchs}
Let $a_i, f_i \in \ov\Q[X]$, $i =1, \ldots, k$,  be   such that $f_i/f_j\not\in \ov\Q^*$ for all $1\le i< j\le k$. Let  polynomials $F_n$, $n\ge 1$, be defined by~\eqref{eq:Fn}. Then there exists an effectively computable constant $C_5$ depending on $a_i,f_i$, $i=1,\ldots,k$, such that if $F_n=0$ then $n < C_5$.
\end{lemma}

We also have the following lower bound  on the number of distinct zeros of $F_n(X)$ in~\eqref{eq:Fn}.

\begin{lemma}  
\label{lem:NmbrZerosFn}
Let $a_i, f_i \in \ov\Q[X]$, $i =1, \ldots, k$,  be  as in Theorem~\ref{thm:LRS-Zeros Qbar} and such that  $\gcd(a_1f_1, \ldots, a_kf_k)=1$. 
 Then there exists an effective positive constant $C_6$ depending only on $a_i,f_i$, $i=1,\ldots,k$, such that for any $n\ge C_6$,  the polynomial 
 $F_n$ defined by~\eqref{eq:Fn} has at least 
 $$
 \frac{2}{k(k-1)} \max_{i=1,\ldots, n} \{\deg a_i +n\deg f_i\} -2dk
 $$  
  distinct roots. 
  \end{lemma}  
  
  \begin{proof}
  Let $C_5$ be as in Lemma~\ref{lem:fuchs}. Then, for any   $n\ge C_5$ we know that $F_n\ne 0$. We consider now the equation
\begin{equation}
\label{eq:abc}
F_n(X)-\sum_{i=1}^k a_i(X)f_i(X)^n =0.
\end{equation}
We also note that for $n\ge C_5$  no subsum of the terms from the set 
$$
\left\{F_n(X), a_1(X)f_1(X)^n, \ldots, a_k(X)f_k(X)^n\right\}
$$ 
vanishes identically. Indeed, if this happens, then there is
 a  proper subset $\cI\subseteq \{1,\ldots,k\}$ such that 
 $$\sum_{i\in \cI} a_i(X)f_i(X)^n=0.
 $$ 
 Since $n\ge C_5$, this  contradicts Lemma~\ref{lem:fuchs}.

Let $\cS_n$ be the set of all distinct zeros of $F_n$. We can now apply Lemma~\ref{lem:abc} to~\eqref{eq:abc} to conclude that 
$$
\max_{i=1,\ldots, k} \{\deg a_i +n\deg f_i\}\le \frac{k(k-1)}{2}\(\# \cS_n +\sum_{i=1}^k (\deg a_i +\deg f_i)\).
$$
From here we obtain
$$
\# \cS_n \ge \frac{2\max_{i=1,\ldots, k} \{\deg a_i +n\deg f_i\}}{k(k-1)} -2dk,
$$
which concludes the proof (with $C_6 = C_5$).
\end{proof} 

\begin{remark}  We note that under a more restrictive condition that the $k$ products $a_1f_1, \ldots, a_kf_k$ 
are pairwise relatively prime, one can take $C_6=1$ in Lemma~\ref{lem:NmbrZerosFn}.  Furthermore, 
under the more restrictive condition that $2k$ polynomials $a_1, f_1, \ldots, a_k, f_k$  are pairwise  relatively prime, one can also use the result of Brindza~\cite[Theorem~1]{Bri} to show that the polynomial $F_n$ has at least $n/(k-1)-(d+1)k$ distinct roots. 
\end{remark}

We also remark that a tight lower bound for the degree of $F_n$   has recently been given in~\cite[Corollary~2]{FuHe}
(in fact  for arbitrary non-degenerate  linear recurrences including non-simple ones). 
      
\subsection{Characterisation of linear recurrence  sequences}
 
To test whether  an arbitrary sequence  $ (v_n)_{n=1}^\infty$ of elements of a field $\F$
is a linear recurrence  sequence, we recall the following well-known result 
which is based on the vanishing of  the
{\sl Kronecker--Hankel
determinants\/} 
\begin{equation}
\label{eq:KH-Det}
\Delta_{h}=\det \(v_{i+j}\)_{0\le i,j \le h-1}.
\end{equation} 

More precisely, by~\cite[Theorem~8.75]{LiNi} (see also~\cite[Theorem~1.6]{EvdPSW} and~\cite[Lemma~5, Chapter~V]{Kob}
for variations) we have: 

\begin{lemma}\label{lem: HK-Det LRS}
A sequence  $(v_n)_{n=1}^\infty$ of elements of a field $\F$
is a   linear recurrence  sequence of order $k$ if and only if for 
the determinants~\eqref{eq:KH-Det} we have $\Delta_{h} =0$ for all  $h\ge k+1$. 
\end{lemma}

\section{Proofs of  Results Towards the Skolem problem}

\subsection{Proof of Theorem~\ref{thm:LRS-Zeros Qbar} }

Multiplying by common  denominators of $a_i(X)$ and   of $f_i(X)$, $i =1, \ldots, k$, we  
can assume that $a_i(X), f_i(X) \in \ov\Q[X]$, $i =1, \ldots, k$. 

Let $\alpha\in\ov\Q \setminus \fE_{\vec{a},\vec{f}}$  be such that~\eqref{eq:zero LRS} holds. 

We now consider the case when we have at least three dominant roots, that is, there exist distinct  integers 
$1 \le r< s < t \le k$ such that 
$$
|f_r(\alpha)|=|f_s(\alpha)|=|f_t(\alpha)|,
$$
or equivalently, 
$$
\frac{|f_s(\alpha)|}{|f_r(\alpha)|}=\frac{|f_t(\alpha)|}{|f_r(\alpha)|}=1.
$$
Since by hypothesis, $(f_s/f_r,f_t/f_r)$ is a non-exceptional rational function, from Lemma~\ref{lem:level curves} we  see that for each of the $k(k-1)(k-2)/6$ possible choices of  the triple $(r,s,t)$ there are at most $4d^2$ such  $\alpha \in  \ov\Q$. 

Hence in total we have excluded at most 
\begin{equation}
\label{eq:Exc Elem}
4d^2 k(k-1)(k-2)/6 = 2d^2 k(k-1)(k-2)/3  
\end{equation} 
elements $\alpha\in\ov\Q \setminus \fE_{\vec{a},\vec{f}}$. 

We thus assume that we have at most two dominant roots. 
In this case we apply  Lemma~\ref{lem: 1-2 dom. root} to any element $\alpha\in\ov\Q\setminus \fE_{\vec{a},\vec{f}}$ and derive that 
$$
n \le \exp\(C_7 D^4_\alpha (h(\alpha)+1)\), 
$$
for an effectively computable constant $C_7$,  
 which depends only on  $a_1,\ldots,a_k,f_1, \ldots, f_k$. 
We now observe that for $n\le C_3$, 
where $C_3$ is as in Corollary~\ref{cor:AMZ}, the result is trivial.
Otherwise  Corollary~\ref{cor:AMZ} applies and we obtain the desired result.

\subsection{Proof of Corollary~\ref{cor:LRS-Zeros AD}}  

By Theorem~\ref{thm:LRS-Zeros Qbar}  we only need to estimate the number $R(D)$ of roots of unity of degree at most $D$ over $\Q$. 
Clearly
$$
R(D) = \sum_{m:~\varphi(m) \le D} \varphi(m)  \le D  \sum_{m:~\varphi(m) \le D} 1,
$$
where $\varphi(m)$ is the Euler function. Using the bound
$$
\sum_{m:~\varphi(m) \le D} 1 \le 23 D
$$ 
of Dubickas and Sha~\cite[Lemma~4.1]{DuSha},  we see that 
\begin{equation}
\label{eq:RD}
R(D) \le 23 D^2
\end{equation} 
(using partial summation one can certainly obtain a tighter bound). 
 Now, since each of $k(k-1)/2$ ratios $f_i(\alpha)/f_j(\alpha)$ can be a root of unity of degree at most $D$,
 there are at most 
 $$
 dR(D) \frac{k(k-1)}{2} < 12 dD^2k^2
 $$ 
 such $\alpha\in\cA_D$, which we need to exclude.

Taking into account that we also need to exclude  at most other $2d^2 k(k-1)(k-2)/3$ elements  $\alpha \in  \cA_D$ as in Theorem~\ref{thm:LRS-Zeros Qbar}, 
and at most $2dk$ zeros  of  $a_if_i$, $i=1,\ldots,k$, and thus  at most $$
2d^2 k(k-1)(k-2)/3+2dk \le 2d^2k^3/3
$$
elements (elementary calculus show that the last inequality holds for any $d \ge 1$ and $k\ge 2$).
Hence we conclude the proof.

\subsection{Proof of Corollary~\ref{cor:IrredFact}} 
 
Let $\K$ be a number field with $[\K:\Q]=D$ such that $a_i,f_i\in\K[X]$, $i=1,\ldots,k$. 

Let $\alpha \not \in \fE_{\vec{a},\vec{f}}$ be a root of $F_n$ for some $n\ge 1$. Let $G$ be its minimal polynomial over $\K$, which implies that $G\mid F_n$. 
Let $D_G=\deg G=[\K(\alpha):\K]$ and let $\M_\alpha$  be the smallest Galois extension of $\Q$ which includes $\K(\alpha)$.

For any $\sigma\in \Gal(\K(\alpha)/\K)$, conjugating the relation $F_n(\alpha)=0$ gives us
\begin{equation}
\label{eq:alpha Fn}
\sum_{i=1}^k a_i(\sigma(\alpha)) f_i(\sigma(\alpha))^n=0.
\end{equation}

We may assume that $D_G>2d^2k^3/3$. Indeed, if $D_G\le 2d^2k^3/3$, then
$$
[\K(\alpha):\Q]\le DD_G \le \frac{2}{3}  Dd^2k^3, 
$$ 
and since by Corollary~\ref{cor:AMZ}, $h(\alpha)$ is bounded from the above by a constant depending only on  $a_1, f_1, \ldots, a_k, f_k$, by Northcott's Theorem there are finitely many such $\alpha$. Thus we can  exclude the irreducible factors of $F_n$, $n\ge 1$, corresponding to these finitely many elements.

Now, since there are $D_G>2d^2k^3/3$ distinct elements $\sigma(\alpha)$ satisfying~\eqref{eq:alpha Fn}, and taking into account the hypothesis on the polynomials $f_i$, we can apply Theorem~\ref{thm:LRS-Zeros Qbar} to conclude that there exists $\sigma\in\Gal(\K(\alpha)/\K)$ such that~\eqref{eq:alpha Fn} holds and for which 
$$
n\le \exp\(C_8[\M_\alpha:\Q]^4\)  
$$ 
for some constant $C_8$ which depends only on  $a_1, f_1, \ldots, a_k, f_k$. This concludes the proof.

\subsection{Proof of Corollary~\ref{cor:split field}}
Let the set $\cE$ be as in Corollary~\ref{cor:IrredFact}.  Clearly, if $F_n$ has a root outside of  the set $\cE \cup   \fE_{\vec{a},\vec{f}}$ the result is 
instant from Corollary~\ref{cor:IrredFact}.

Otherwise,  since $\cE$ is a finite set and so are the sets of $\alpha\in \ov \Q$ with  
$f_i(\alpha)=0$ or $a_i(\alpha)=0$, for some $1 \le i \le k$, adjusting the constant $c_0$ we can make 
the desired result valid for these values of $\alpha$ as well.

Hence it remains to consider the case when the ratio  $f_i(\alpha)/f_j(\alpha)$  is a root of unity
for some $1 \le i <j\le k$.  We recall that by Lemma~\ref{lem:NmbrZerosFn}  $F_n$ has at least $2n/k(k-1) - C_9$ distinct roots
$\alpha$, for some constant $C_9$ which depends only on  $a_1, f_1, \ldots, a_k, f_k$. 
By our assumption, for each of them at least one ratio  $f_i(\alpha)/f_j(\alpha)$  is a root of unity for some $1 \le i <j\le k$.
For each $\gamma \in \ov \Q$ the equation $f_i(\alpha)/f_j(\alpha) = \gamma$  has at most $d$ solutions
$\alpha \in \ov \Q$.  
Hence the  set of roots of unity among the ratios $f_i(\alpha)/f_j(\alpha)$  form a set $\cU_\vec{f}$ of 
at least     
$$\# \cU_\vec{f} \ge 2n/dk(k-1) - C_9/d
$$ 
 distinct  roots of unity $\rho$ for some $1 \le i <j\le k$.

 Recall the bound~\eqref{eq:RD} on the number of roots of unity of degree at most $D$. 
Hence at least one of  the above  roots of unity 
 is of degree at least 
 $$  \sqrt{ \frac{1}{23}\# \cU_\vec{f}} \ge\sqrt{ \frac{2n}{23dk(k-1)} - C_9/23d}. 
 $$ 
 and thus so is $\alpha$. Hence in this case we have a much stronger bound than claimed. 
 
\section{Proof of Theorem~\ref{thm:LRS-PowU} on Perfect Powers in Specialisations at Roots of Unity} 

Let $\K$ be the  field of definition  of the polynomials $a_i, f_i$. Thus, $a_i,f_i\in\Z_\K[X]$, where $\Z_\K$ is the ring of integers of $\K$.

As we have shown in the proof of Theorem~\ref{thm:LRS-Zeros Qbar},  see~\eqref{eq:Exc Elem}, for all but at most $2d^2k^3/3$ 
elements $\alpha \in   \U \setminus \fE_{\vec{a},\vec{f}}$, the sequence
$\(F_n(\alpha)\)_{n=1}^{\infty}$ has at most two dominant roots.

Let us assume first that for such an element $\alpha \in   \U \setminus \fE_{\vec{a},\vec{f}}$, the sequence $\(F_n(\alpha)\)_{n=1}^{\infty}$ has two dominant roots, that is, for some $1\le i\ne j\le k$, $|f_i(\alpha)|=|f_j(\alpha)|$. This implies that 
\begin{equation}
\label{eq:eq rou}
f_i(\alpha)\ov{f_i}(\ov\alpha) =  f_i(\alpha)\ov{ f_i(\alpha)} =f_j(\alpha)\ov {f_j(\alpha) } =f_j(\alpha)\ov{f_j}(\ov\alpha),
\end{equation}
where $\ov\alpha$ is the complex conjugate of $\alpha$, which is again a root of unity. 

Since by our assumption, 
$$f_i(X)\ov{f_i}(Y) - f_j(X)\ov{f_j}(Y)\in\K[X,Y], \qquad 1\le i\ne j\le k,
$$ 
do not have any factor of the form $X^rY^s-u$ or $X^r-uY^s$ with  $u \in \U$, 
 by a  result of Lang~\cite{Lang}  we know that there are finitely many solutions $(\alpha,\beta)$ in roots of unity (and thus also of the form $(\alpha,\ov\alpha)$) to the equation~\eqref{eq:eq rou}. Moreover, by~\cite[Section~4.1]{BS}, the equation 
$$f_i(X)\ov{f_i}(Y) - f_j(X)\ov{f_j}(Y)=0
 $$ has at most  $44d^2$ solutions $(\alpha,\beta)$ in roots of unity. Thus, there are at most 
 $22d^2k(k-1)< 22d^2k^2$ such elements $\alpha \in   \U \setminus \fE_{\vec{a},\vec{f}}$ such that the sequence $\(F_n(\alpha)\)_{n=1}^{\infty}$ has  two dominant roots.

Thus, for all but $d^2(2k^3/3 + 22k^2)$ elements $\alpha \in   \U \setminus \fE_{\vec{a},\vec{f}}$, the sequence $\(F_n(\alpha)\)_{n=1}^{\infty}$ has only one dominant root. Let such an element $\alpha \in   \U \setminus \fE_{\vec{a},\vec{f}}$. 
We now apply  the result of Fuchs~\cite[Corollary~2.10]{Fuchs} describing the structure of the set of perfect powers in 
linear recurrent sequences over number fields (one easily verifies that for $\alpha \in   \U \setminus \fE_{\vec{a},\vec{f}}$ 
all necessary conditions of~\cite[Corollary~2.10]{Fuchs}  are satisfied).  This concludes the proof.

\section{Proofs of  Results on Common Zeros}
 
\subsection{Proof of Theorem~\ref{thm:gcd prod pow}} 
For given $\m=(m_1,\ldots,m_r)$,   $\n=(n_1,\ldots,n_s)$,  and $u,v\in\Gamma^\di$, we define
$$
\cD_{\m,\n,u,v}(X)=\gcd\(\prod_{i=1}^rf_i(X)^{m_i}-u,\prod_{j=1}^sg_j(X)^{n_j}-v \).
$$ 
We recall that by the greatest common divisor of two rational functions we mean the greatest monic common divisor of their numerators. In particular, $\cD_{\m,\n,u,v}\in\ov\Q[X]$.

The proof follows the same approach as in the proof of~\cite[Theorem~1.3]{Ost}. Indeed, since the rational functions $f_1,\ldots,f_r,g_1,\ldots,g_s$ are multiplicatively independent modulo $\Gamma$, by Lemma~\ref{lem:multdep-rat}, applied with $n=r+s$ and
\begin{align*} 
&(k_1,\ldots,k_n)= (m_1,\ldots,m_r, 0, \ldots, 0),\\
&(\ell_1,\ldots,\ell_n) = (0, \ldots, 0, n_1,\ldots,n_s), 
\end{align*} 
  there are finitely many $\alpha\in \ov\Q$ such that $\cD_{\m,\n,u,v}(\alpha)=0$ for some non-zero integer vectors $\m,\n$ as above and $u,v\in\Gamma^\di$. 
  We denote this finite set, depending only on $f_i,g_j$, by $\cS$.

To construct the polynomial $h$ we need to control also the multiplicity of the roots of the polynomials $\cD_{\m,\n,u,v}$. This is also readily given by Lemma~\ref{lem:ABCrat}, which implies that 
$$
\Mult(\cD_{\m,\n,u,v}) \le 2\min\left\{\sum_{i=1}^r\deg f_i ,\, \sum_{j=1}^s \deg g_j\right\}.
$$
Thus, we can define the polynomial $h\in\ov\Q[X]$ by
$$
h(X)=\prod_{\gamma\in\cS} (X-\gamma)^{2\min\left\{\sum_{i=1}^r\deg f_i ,\, \sum_{j=1}^s \deg g_j\right\}},
$$
which concludes the proof.
 
\subsection{Proof of Corollary~\ref{cor:gcd poly}} 
Let $F,G\in\ov\Q[X]$ be of degree at most $d\ge 1$ with all roots in $\Gamma^\di$. We reduce the problem to looking at each greatest common divisor
$$
\gcd\(\prod_{i=1}^rf_i(X)^{m_i}-\gamma_1,\prod_{j=1}^sg_j(X)^{n_j}-\gamma_2\)
$$
with roots $\gamma_1$ and $\gamma_2$ of $F$ and $G$ respectively
and $(m_1,\ldots,m_r)\in\N^r\setminus\{{\mathbf 0}\}$,  $(n_1,\ldots,n_s)\in\N^s\setminus\{{\mathbf 0}\}$.

By Theorem~\ref{thm:gcd prod pow} there exists a polynomial $h\in\ov\Q[X]$ that depends only on $f_i,g_j$ and the generators of $\Gamma$ such that for all 
$(m_1,\ldots,m_r)\in\N^r\setminus\{{\mathbf 0}\}$,  $(n_1,\ldots,n_s)\in\N^s\setminus\{{\mathbf 0}\}$ and $\gamma_1,\gamma_2\in\Gamma^\di$ one has
$$
\gcd\(\prod_{i=1}^r f_i(X)^{m_i}-\gamma_1,\prod_{j=1}^sg_j(X)^{m_j}-\gamma_2\)\mid h.
$$ 
Since both $F$ and $G$ have at most $d$ roots in $\Gamma^\di$, we conclude the proof choosing $H=h^{d^2}$. 
 
\section{Proof of Theorem~\ref{thm:LRS spec inf}}   

For each $\alpha\in\C$ we consider the sequence $\(\Delta_h(\alpha)\)_{h=0}^{\infty}$ of Kronecker-Hankel determinants 
$$
\Delta_h(\alpha)=\det\(F_{i+j}(\alpha)\)_{0\le i,j\le h-1}, \quad h=0,1,\ldots.
$$

Assume there exist infinitely many $\alpha\in\C$ such that  $\(F_n(\alpha)\)_{n=0}^{\infty}$ is a linear recurrence of order at most $K$. Hence, 
be Lemma~\ref{lem: HK-Det LRS} 
for all  $h\ge K+1$ we have
$$
\Delta_h(\alpha)=0 
$$
for infinitely many $\alpha\in\C$. Therefore, for  all  $h\ge K+1$, the rational function 
$$
\Delta_h(X)=\det\(F_{i+j}(X)\)_{0\le i,j\le h-1} 
$$
has infinitely many zeros, and thus is identical to zero.

Applying  Lemma~\ref{lem: HK-Det LRS}  again, we conclude the proof.

\section{Comments and Further Questions} 
\label{sec:comm}

As we have mentioned Theorem~\ref{thm:LRS-Zeros Qbar}  can be extended to non-simple sequences
without new ideas and just at the cost of introducing more complicated notations.

We also remark that Theorem~\ref{thm:LRS-Zeros Qbar} is an analogue of a result 
of Kulkarni, Mavraki, and Nguyen~\cite[Proposition~2.2]{KMN}, 
see also~\cite[Proposition~2.2]{BDZ}, in which the coefficients $a_1, \ldots, a_k$ of $F_n(X)$ 
in~\eqref{eq:Fn} are constants rather than rational functions as in our case. 
Moreover,~\cite[Proposition~2.2]{KMN} is 
not effective while Theorem~\ref{thm:LRS-Zeros Qbar} is. 

Furrhermore, if as in~\cite{KMN} the coefficients $a_1, \ldots, a_k$ of $F_n(X)$ 
in~\eqref{eq:Fn} are constants then analysing the bounds of~\cite{Sha} underlying 
Lemma~\ref{lem: 1-2 dom. root} we see that $D^4$ can be replaced with $D_0^3D$, 
where $D_0$ is the degree of the Galois closure of $\Q(a_1, \ldots, a_k)$ over $\Q$. In turn, in this case, this leads to a single exponential bound of the form 
$$
n \le \exp\(C D_\alpha\),
$$
in Theorem~\ref{thm:LRS-Zeros Qbar}.

We describe a possible generalisation of Theorem~\ref{thm:LRS-Zeros Qbar} to $S$-unit 
equations in $\ov\Q(X)$. Let  $\Gamma$  be a finitely generated subgroup of $\ov\Q(X)$ 
and fix rational functions $a_1,\ldots,a_k\in\ov\Q(X)$. By~\cite[Proposition~6.1]{AMZ}, 
for any rational functions $u_1,\ldots,u_k\in\Gamma$ such that  
$$
u_i/u_j \not \in \ov \Q, \quad 1\le i <j\le k, \mand \sum_{i=1}^k a_iu_i\ne 0,
$$ 
the set of $\alpha\in\ov\Q$  such that  
\begin{equation}
\label{eq:S-unit eq}
\sum_{i=1}^k a_i(\alpha)u_i(\alpha)=0
\end{equation}
is a set of bounded height, depending only on $a_1,\ldots,a_k$ and the generators of $\Gamma$. 

We now ask for an analogue of Theorem~\ref{thm:LRS-Zeros Qbar} for such equations. It is convenient 
to define the notion of a {\it primitive solution\/} to~\eqref{eq:S-unit eq} as a solution with $u_i(\alpha) = 1$ 
for some $i = 1, \ldots, k$. 

We now ask the following. 

\begin{question} 
\label{quest:S-unit} 
Is it true, under some natural conditions on the generators of $\Gamma$,  that outside of  
a set of bounded height of values $\alpha\in\ov\Q$,  then for 
every primitive solution to~\eqref{eq:S-unit eq}, $\max_{i=1,\ldots,k}\deg u_i$ is bounded only in terms of the degree $D_\alpha$ of the smallest Galois field $\K$ over $\Q$ with $\alpha \in \K$, the functions $a_1,\ldots,a_k$ and the generators of $\Gamma$?
\end{question}

We note that the idea of the proof of~\cite[Proposition~6.1]{AMZ} which reduces $S$-unit equations to 
equations of the type~\eqref{eq:zero LRS}, can perhaps help to tackle Question~\ref{quest:S-unit}. 
 Unfortunately, during this reduction we do not control
well the corresponding polynomials $a_1,f_1,\ldots,a_k,f_k \in\ov\Q(X)$ and in particular it is 
not immediately clear how to verify the necessary condition of Theorem~\ref{thm:LRS-Zeros Qbar}.

We remark that the proof of  Theorem~\ref{thm:LRS-Zeros Qbar}  relies on the fact  that outside of a small set of parameters the corresponding 
specialisations are sequences with at most two dominant roots. 
On the other hand, in the proof  of Theorem~\ref{thm:LRS-PowU} we show that for all but finitely many 
specialisations at roots of unity these sequences have only one dominant root.  These ideas can be used to study many 
other properties of the corresponding linear recurrence sequences. For example, by combining this approach with results and ideas 
of~\cite{Shp,Stew1,Stew2} one can study prime ideal divisors of elements of these sequences.

\section*{Acknowledgement}

The authors are very grateful to the referee for the very helpful comments.

This work  was  supported, in part,  by the Australian Research Council Grants~DP180100201 and DP200100355. The authors are grateful to Umberto Zannier for useful discussions on previous versions of the file. The first author gratefully acknowledges the hospitality and generosity of the Max Planck Institute of Mathematics, where parts of this paper were developed.


\begin{thebibliography}{99}

\bibitem{AR} N. Ailon and Z. Rudnick, `Torsion points on curves and common divisors of $a^k-1$ and $b^k-1$',  
{\it Acta Arith.\/},  {\bf 113} (2004),   31--38. 

 
 \bibitem{AMZ} 
F. Amoroso, D. Masser and U. Zannier, 
`Bounded height in pencils of finitely generated subgroups', 
 {\it Dule Math. J.\/}, {\bf 166} (2017) 2599--2642. 
 
 \bibitem{AmVia}
F. Amoroso and E. Viada, `On the zeros of linear recurrence sequences', 
{\it Acta Arith.\/}, {\bf  147} (2011), 387--396.

\bibitem{BDZ}
J. P. Bell, K. D. Nguyen and U. Zannier, 
`$D$-finiteness, rationality, and height', 
 {\it Trans. Amer. Math. Soc.\/},  {\bf  373} (2020),  4889--4906. 
   
   \bibitem{BeMi} J. Berstel and M. Mignotte, `Deux propri\' et\' es d\' ecidables des suites r\' ecurrentes
lin\' eaires', {\it Bull. Soc. Math. France}, {\bf 104} (1976), 175--184. 

\bibitem{BS} F. Beukers and C. J. Smyth, `Cyclotomic points on curves', 
{\it Number Theory for the
Millenium} (Urbana, Illinois, 2000), vol.~I, A.~K.~Peters, 2002, 67--85. 

\bibitem{BiLu}
Y.   Bilu and F. Luca, `Binary polynomial power sums vanishing at roots of unity', 
{\it  Acta Arith.\/},   to appear. 

\bibitem{BHMZ} E. Bombieri, P. Habegger, D. Masser and U. Zannier, \textit{A Note on Maurin Theorem}, Rend. Lincei Mat. Appl. {\bf 21} (2010), 251--260. 
 
 \bibitem{BMZ} E. Bombieri, D. Masser and U. Zannier, `Intersecting a curve with algebraic subgroups of multiplicative groups', {\it Int. Math. Res. Not.}, {\bf 20} (1999), 1119--1140.


 \bibitem{BMZ08} E. Bombieri, D. Masser and U. Zannier, `On unlikely intersections of complex 
 varieties with tori', {\it Acta Arith.}, {\bf 133} (2008), 309--323.
 
 \bibitem{BoGu} E. Bombieri and W. Gubler, {\it Heights in Diophantine geometry},   Cambridge Univ. Press, Cambridge, 2006.
 
 \bibitem{Bri} B. Brindza, `Zeros of polynomials and exponential diophantine equations', {\it Compos. Math.\/}, {\bf 61} (1987), 137--157.
 
 \bibitem{BrMa} W. D. Brownawell, D. Masser, `Vanishing sums in function fields', {\it Math. Proc. Cambridge Philos. Soc.}, {\bf 100} (1986), 427--434. 
 
 \bibitem{BCZ} Y. Bugeaud, P. Corvaja and U. Zannier, `An upper bound for the G.C.D. of $a^n-1$ and $b^n-1$', {\it Math. Z.}, {\bf 243} (2003), 79--84.

\bibitem{CZ05}  P. Corvaja and U. Zannier, `A lower bound for the height of a rational function at $S$-unit points', {\it Monatsh. Math.}, {\bf 144} (2005), 203--224.
 
  \bibitem{CZ} P. Corvaja and U. Zannier, `Some cases of Vojta's conjecture on integral points over function
fields', {\it J. Alg. Geometry}, {\bf 17} (2008), 295--333.

\bibitem{DuSha} 
A. Dubickas and M. Sha,  `The distance to square-free polynomials',  
{\it Acta Arith.\/},  {\bf  186} (2018),  243--256.
 
\bibitem{EvdPSW} G. Everest, A. J. van der Poorten, I. E. Shparlinski 
and T. B. Ward, {\it Recurrence sequences\/}, Amer. Math. Soc., 2003.


\bibitem{Fuchs} C. Fuchs, `Polynomial-exponential equations and linear recurrences',  
{\it Glasnik Matem.\/}, {\bf 38} (2003), 233--252. 

\bibitem{FuHe} C. Fuchs and S. Heintze, `On the growth of linear recurrences in function fields', {\it Preprint\/}, 2020, available at \url{https://arxiv.org/abs/2006.11074}. 


\bibitem{FuPe} C. Fuchs,  and A. Peth\"o, `Effective bounds for the zeros of linear recurrences in function fields', {\it Journal de Th\' eorie des Nombres de Bordeaux\/},  {\bf  17} (2005), 749--766. 


\bibitem{HinSil} M. Hindry and J. H. Silverman, {\it Diophantine geometry\/}, Graduate Texts
in Mathematics {\bf 201}, Springer, New York, 2000.  

\bibitem{Kob}
N. Koblitz, {\it $P$-adic numbers, $p$-adic analysis, and zeta-functions\/}, Springer-Verlag, New York, 1977.
 
 \bibitem{KMN}
 A. Kulkarni, N. M. Mavraki, and K. D. Nguyen, 
 `Algebraic approximations to linear combinations of powers: an extension of results by Mahler and Corvaja-Zannier', 
 {\it Trans. Amer. Math. Soc.\/}, (to appear). 
 
\bibitem{Lang} S. Lang, `Division points on curves', {\it Ann. Mat. Pura Appl.\/},  {\bf 70} (1965), 229--234.
 
\bibitem{Lev} A. Levin, `Greatest common divisors and Vojta?s conjecture for blowups 
of algebraic tori', {\it Invent. Math.\/},  {\bf 215} (2019), no. 2, 493--533.
 
 \bibitem{LevWan} A. Levin and J. T.-Y. Wang, `Greatest common divisors of analytic functions and Nevanlinna theory of algebraic tori', 
 {\it J. Reine Angew. Math.\/},   (to appear). 
 
 \bibitem{LiNi} R. Lidl and H. Niederreiter,  {\it Finite Fields\/}, 
Cambridge Univ. Press, Cambridge, 1997.

 \bibitem{Mas0}  R. C. Mason, {\it Diophantine equations over function fields\/},  London Math. Soc. Lecture Note Series {\bf 96}, Cambridge Univ. Press, Cambridge, 1984

\bibitem{Mau} G. Maurin, `Courbes alg\' ebriques et \' equations multiplicatives', {\it Math. Ann.\/},  {\bf 341} (2008), 789--824.

\bibitem{Ost} A. Ostafe, `On some extensions of the Ailon-Rudnick theorem', 
{\it Monat. f\" ur Math.\/},  {\bf 181} (2016), 451--471. 

\bibitem{PaSh}
F. Pakovich and I. E. Shparlinski, 
`Level curves of rational functions and unimodular points on rational curves', 
 {\it Proc. Amer. Math. Soc.\/},  {\bf  148} (2020),  1829--1833.
  
 \bibitem{Sha}
M. Sha,  `Effective results on the Skolem problem for linear recurrence sequences', 
 {\it J. Number Theory\/}, {\bf 197} (2019) 228--249.  
 

 \bibitem{Shp} I. E. Shparlinski, `Prime divisors of recurrent sequences',
{\it  Isv. Vyssh. Uchebn. Zaved. Math\/}, no.4, (1980), 101--103
 
 \bibitem{Si84} J. H. Silverman, `The S-unit equation over function fields',
{\it Proc. Camb. Philos. Soc.\/}, {\bf 95} (1984), 3--4.
 
 \bibitem{Silv}
J. H. Silverman, {\it The arithmetic of dynamical systems\/}, 
Springer-Verlag, New York, 2007.

\bibitem{Stew1} C. L. Stewart,
`On divisors of terms of linear recurrence sequences',
{\it J. Reine Angew. Math.\/}, {\bf 333} (1982),  12--31.

 \bibitem{Stew2} C. L. Stewart,  `On the greatest square-free factor of terms of a linear recurrence sequence', 
 {\it  Diophantine Equations\/}
Tata Inst. Fund. Res. Stud. Math., 20, Tata Inst. Fund. Res., Mumbai, 2008,  257--264,

\bibitem{St} W. W. Stothers, `Polynomial identities and Hauptmoduln',
{\it Quarterly J. Math. Oxford\/}, {\bf 32} (1981), 349--370.

\bibitem{Vo} J. F. Voloch, `Diagonal equations over function fields', {\it Bol. Soc. Brasil. Mat.\/},  {\bf 16} (1985), 29--39. 

\bibitem{Zan} U.  Zannier, 
{\it Lecture notes on Diophantine analysis\/},
Publ. Scuola Normale Superiore, Pisa, 2009. 
 
\end{thebibliography}
\end{document}